\documentclass[12pt,a4paper]{amsart}
\usepackage[utf8]{inputenc}
\usepackage{amsmath}
\usepackage{amsthm}
\usepackage{amssymb}
\usepackage{wasysym}
\usepackage{graphicx}
\usepackage{geometry}
\usepackage{xcolor}
\usepackage[unicode]{hyperref}
\hypersetup{
	colorlinks,
    linkcolor={red!60!black},
    citecolor={green!60!black},
    urlcolor={blue!60!black},
}
\usepackage[abbrev, msc-links]{amsrefs}
\def\N{\mathbb N}
\def\X{\mathcal X}
\def\Y{\mathcal Y}
\def\G{\mathcal G}
\def\H{\mathcal H}
\def\V{\mathcal V}
\def\C{\mathcal C}
\def\D{\mathcal D}
\def\M{\mathcal M}

\newtheorem{theorem}{Theorem}[section]
\newtheorem{lemma}[theorem]{Lemma}
\newtheorem{corollary}[theorem]{Corollary}
\newtheorem{remark}[theorem]{Remark}
\newtheorem{conjecture}[theorem]{Conjecture}

\title{Universal graphs with forbidden wheel minors}
\author{Thilo Krill}
\address{Universit\"at Hamburg, Department of Mathematics, Bundesstrasse 55 (Geomatikum), 20146 Hamburg, Germany}
\email{thilo.krill@uni-hamburg.de}
\keywords{infinite graphs; universal graphs; forbidden minor; wheel; cycle}

\begin{document}

\begin{abstract}
    Let $W$ be any wheel graph and $\mathcal{G}$ the class of all countable graphs not containing $W$ as a minor. We show that there exists a graph in $\mathcal{G}$ which contains every graph in $\mathcal{G}$ as an induced subgraph.
\end{abstract}

\maketitle

\section{Introduction}

A graph $\Gamma$ is \emph{weakly universal} in a class of graphs $\G$ if $\Gamma\in\G$ and each $G\in\G$ is isomorphic to a subgraph of $\Gamma$. If additionally each $G\in\G$ is isomorphic to an induced subgraph of $\Gamma$, then we say that $\Gamma$ is \emph{strongly universal}.

Our main result reads as follows:
\begin{theorem}\label{thm_main_universal}
There exists a strongly universal graph in the class of all (countable) $X$-minor-free graphs if $X$ is one of the following:
\begin{itemize}
\item a cycle,
\item the union of two cycles that intersect in exactly one edge, or
\item a wheel, i.e. a graph consisting of a cycle together with an additional vertex which is adjacent to all vertices on the cycle.
\end{itemize}
\end{theorem}
\noindent The case where $X$ is a cycle has previously been proven by Komjáth, Mekler, and Pach \cite{komjathmeklerpach}, and we provide an alternative proof.

Universal graphs are a well studied field in infinite graph theory.
For instance, Füredi and Komjáth characterised for which finite 2-connected graphs $X$ there exists a strongly or weakly universal graph without $X$ as a subgraph \cite{furedikomjath}. Cherlin, Shelah, and Tallgren characterised the same when $X$ is a tree \cites{cherlintallgren,cherlinshelah}.

In contrast, very few is known about the existence of universal $X$-minor-free graphs for graphs $X$ which are not listed in Theorem \ref{thm_main_universal}. If $X$ is a path or a path with an adjoined edge, then the graphs without $X$ as a minor are precisely those without a subgraph isomorphic to $X$, and by \cite{cherlintallgren} there is a strongly universal graph in this class. Other than that, the existence of universal $X$-minor-free graphs is so far only known for some graphs $X$ on few vertices (see \cite{diestelsurvey}). For example, it follows from structure theorems of Halin and Wagner that a weakly universal $X$-minor-free graph exists when $X$ is the graph $K_{2,3}$ \cite{halinwagner}, or $K_5$ minus an edge \cite{wagner2}, or the Cartesian product of $C_3$ with $K_2$ \cite{halin4}. For results on a different notion of universality in minor closed graph classes, also see \cite{diestelkuhn, georgakopoulos}.

As for negative results, Diestel, Halin, and Vogler \cite{diestelhalinvogler} showed that there is no weakly universal $X$-minor-free graph if $X=K_r$ for $r\geq 5$ or if $X$ contains a ray. More recent progress was made by Huynh, Mohar, Šámal, Thomassen, and Wood \cite{huynh}, who showed that every countable graph containing a copy of every countable planar graph as a subgraph must contain an infinite clique minor.
Therefore, there cannot exist a weakly universal $X$-minor-free graph for any countable non-planar $X$: Since all planar graphs are $X$-minor free, this universal graph would contain copies of all countable planar graphs. Therefore, it would contain an infinite clique and in particular $X$ as a minor.

Interestingly, all known results on the topic, including those presented in this paper, lend support to following conjecture:
\begin{conjecture}
The following are equivalent for any finite connected graph $X$:
\begin{itemize}
    \item there is a strongly universal $X$-minor-free graph,
    \item there is a weakly universal $X$-minor-free graph,
    \item $X$ is planar.
\end{itemize}
\end{conjecture}

Finally, we provide a brief overview of the proof of Theorem \ref{thm_main_universal}, which can be divided into two parts. On the one hand, we prove the following statement for $k\in\{2,3\}$:
\begin{itemize}
\item[$(*)$] For every finite $k$-connected graph $X$ which is a minor in every $k$-connected graph containing arbitrarily long paths, there exists a strongly universal $X$-minor-free graph.
\end{itemize}

To prove $(*)$ in the case $k=2$, we decompose $X$-minor-free-graphs into their 2-connected blocks and then use the absence of long paths in the blocks to construct a universal graph.
A result of Cherlin and Tallgren \cite{cherlintallgren} on universal graphs with excluded paths will be helpful for this.
In the case $k=3$ we use a similar approach, but instead of considering 2-connected blocks, we apply a result of Richter \cite{richter}, which generalises a result of Tutte \cite{tutte} to infinite graphs, and which allows us to decompose a graph into 3-connected graphs and cycles.

On the other hand, we prove that the graphs $X$ listed in Theorem \ref{thm_main_universal} satisfy the conditions in $(*)$ for either $k=2$ or $k=3$.
In the case where $X$ is a wheel, this proof relies on results from \cite{oporowskioxleythomas} by Oporowski, Oxley, and Thomas, where they investigate unavoidable minors of 3-connected graphs.

\section{Preliminaries}\label{preliminaries}

Any graph-theoretic notation that is not defined here can be found in \cite{diestelbook}. All graphs in this paper are countable and from now on we call a strongly universal graph simply a \emph{universal graph}.

\subsection{Minors}

A graph $G'$ is a \emph{model} of a graph $G$ if there is a partition $\{B_v:v\in V(G)\}$ of $V(G')$ into non-empty connected sets of vertices, such that for all $v\neq w\in V(G)$ there is a $B_v$--$B_w$ edge in $G'$ if and only if $vw\in E(G)$. We call the sets $B_v$ \emph{branch sets} and say that $G$ is a \emph{minor} of $H$ if $H$ contains a model of $G$ as a subgraph.

\subsection{$X$-free graphs}

Let $\X$ be any class of graph. We say that a graph $G$ is \emph{$\X$-minor-free} if it does not contain a member of $\X$ as a minor and \emph{$\X$-free} if it does not contain a subgraph isomorphic to a member of $\X$. If $\X$ consists of a singe graph $X$, we also say that $G$ is $X$-minor-free or $X$-free, respectively.

\subsection{Vertex- and edge-colourings}

Let $G$ be a graph and $c,d\in\N$. A \emph{$c$-edge-colouring of $G$} is a map $f:E(G)\to\{0,\ldots,c-1\}$ and a \emph{$d$-vertex-colouring of $G$} is a map $g:V(G)\to\{0,\ldots,d-1\}$. We refer to $f(e)$ or $g(v)$ as the \emph{colour} of $e$ or $v$ in $G$, respectively. We call a graph $G$ together with a $c$-edge-colouring of $G$ a \emph{$c$-graph}, and $G$ together with both a $c$-edge-colouring and a $d$-vertex-colouring of $G$ a \emph{$(c,d)$-graph}. For a class of graphs $\G$, we denote by $\G^c$ the class of all $c$-edge-coloured versions of the graphs in $\G$ and by $\G^{(c,d)}$ the class of all $c$-edge-coloured and at the same time $d$-vertex-coloured versions of the graphs in $\G$. Isomorphisms and (induced) subgraphs in the context of $c$-graphs and $(c,d)$-graphs are required to preserve colourings.

\subsection{Tree-decompositions}

Let $G$ be a graph, $T$ a tree, and $\V=(V_t)_{t\in V(T)}$ a family of subsets of $V(G)$. The pair $(T,\V)$ is a \emph{tree-decomposition} of $G$ if
\begin{itemize}
\item $G=\bigcup_{t\in V(T)}G[V_t]$, and
\item for all $x,y,z\in V(T)$ such that $y$ lies on the $x$--$z$ path in $T$, we have $V_x\cap V_z\subseteq V_y$.
\end{itemize}
We call the sets $V_t$ \emph{parts} of $(T,\V)$ and sets of the form $V_t\cap V_u$ for adjacent $t,u\in V(T)$ \emph{adhesion sets} of $(T,\V)$. The \emph{torso} of a part $V_t$ is obtained from the graph $G[V_t]$ which $V_t$ induces in $G$ by adding edges to make all adhesion sets contained in $G[V_t]$ complete. The number $\max\{|V_t|-1:t\in V(T)\}$ is the \emph{width} of $(T,\V)$ and the maximum size of an adhesion set of $(T,\V)$ is called the \emph{adhesion} of $(T,\V)$.

\section{Lemmas on tree-decompositions}

In this section, we record two simple lemmas on tree-decompositions for later use.

\begin{lemma}\label{longpath}
For all $w\in\N$ there is a function $\ell_w:\N\to\N$ with the following property: Let $G$ be any graph with a tree-decomposition $(T,\V)$ of width less than $w$ and let $k\in\N$. Then $T$ contains a path of length $k$ if $G$ contains a path of length $\ell_w(k)$.
\end{lemma}

\begin{proof}
Let $\ell_w(1)=1$ and $\ell_w(k):=(w+1)\ell_w(k-1)+2w$ for all $k>1$. We use induction on $k$. For $k=1$, the claim is clearly true. Therefore, suppose that $k>1$ and let $G$ be a graph with a tree-decomposition $(T,\V)$ of width less than $w$ such that $G$ contains a path $P$ of length $\ell_w(k)$. Since $T$ is a tree, there is a node $t\in V(T)$ which lies on every longest path in $T$. As $|V_t|\leq w$, there are at least $\ell_w(k)-2w=(w+1)\ell_w(k-1)$ edges in the graph $P-V_t$ and there are at most $w+1$ components of $P-V_t$.
Therefore, there is a component $P'$ of $P-V_t$ which is a path of length at least $\ell_w(k-1)$.

For every component $C$ of $T-t$, we consider the corresponding set of vertices $V_C:=\bigcup_{t\in V(C)}V_t$ in $G$. It is a standard fact about tree-decompositions that $V_t$ separates $V_C$ from $V_{C'}$ in $G$ for all components $C\neq C'$ of $T-t$. Therefore, there is a component $C$ of $T-t$ such that $V(P')\subseteq V_C$. It can be easily checked that $(C,\V')$ is a tree-decomposition of $P'$, where $V'_t:=V_t\cap V(P')$ for all $t\in V(C)$. By the induction hypothesis, $C$ contains a path $Q$ of length $k-1$. However, $Q$ cannot be a longest path in $T$ because $Q$ does not contain $t$. Hence $T$ contains a path of length $k$.
\end{proof}

\begin{lemma}\label{lem:minor_lives_in_part}
Let $G,H$ be graphs and $k\in\N$ such that $H$ is $k$-connected and $H$ is a minor of $G$. Furthermore, let $(T,\V)$ be a tree-decomposition of $G$ of adhesion $<k$ such that all adhesion sets are complete in $G$. Then there is a node $t\in V(T)$ such that $H$ is a minor of $G[V_t]$.
\end{lemma}

\begin{proof}
For every edge $tu$ of $T$, let $C_{(t,u)}$ be the component of $T-tu$ containing $u$ and write $U_{(t,u)}:=\bigcup_{t\in C_{(t,u)}} V_t$ for the corresponding subgraph of $G$.

Fix a model $H'$ of $H$ in $G$ with branch sets $B_x$ for $x\in V(H)$. We begin by showing that for every edge $t_1t_2$ of $T$, there exists a number $i\in\{1,2\}$ such that all branch sets of $H'$ meet $U_{(t_{1-i},t_i)}=:U_i$.
Suppose that there exists no such $i$; then there are vertices $x_1,x_2\in V(H)$ such that $B_{x_1}\subseteq U_1\setminus U_2$ and $B_{x_2}\subseteq U_2\setminus U_1$. Since $H$ is $k$-connected, there are $k$ internally disjoint $x_1$--$x_2$ paths in $H$ and hence $k$ internally disjoint $B_{x_1}$--$B_{x_2}$ paths in $H'\subseteq G$. However, $V_{t_1}\cap V_{t_2}$ contains less than $k$ vertices and separates $U_1\setminus U_2$ from $U_2\setminus U_1$ in $G$, a contradiction.

Now we orient every edge $tu$ of $T$ towards an endvertex $u$ with the property that all branch sets of $H'$ meet $U_{(t,u)}$. We will show that there is no directed ray $t_1t_2t_3\ldots$ in $T$ (we call a path or ray $t_1t_2t_3\ldots$ directed if it orients each edge $t_it_{i+1}$ towards $t_{i+1}$).
Then there exists a maximal directed path in $T$ and its last vertex $t$ has the property that all edges of $T$ incident with $t$ are directed towards $t$. From that it is easy to conclude that all branch sets of $H'$ meet $V_t$. Thus $H$ is a minor of $G[V_t]$ with branch sets $B_x\cap V_t$ for $x\in V(H)$: Indeed, there is an edge in $G[V_t]$ between $B_x\cap V_t$ and $B_y\cap V_t$ whenever there is an edge between $B_x$ and $B_y$ because the adhesion sets of $(T,\V)$ are complete.

It is left to show that $T$ contains no directed ray. Suppose for a contradiction that $R=t_1t_2t_3\ldots$ is a directed ray in $T$.

First, we show that
\begin{itemize}
\item[$(*)$] for every vertex $x\in V(H)$ there is an integer $f(x)$ such that for all $i\geq f(x)$, the part $V_{t_i}$ meets $B_x$.
\end{itemize}

Since the edge $t_1t_2$ is oriented towards $t_2$, there is a node $t\in C_{(t_1,t_2)}$ with $V_t\cap B_x\neq\emptyset$. Let $P$ be the (undirected) $t$--$R$ path in $T$ and let $f(x)\in\N$ such that $t_{f(x)}$ is the endvertex of $P$ in $R$. Next, consider any $i>f(x)$ and let $u\in C_{(t_{i-1},t_i)}$ with $V_u\cap B_x\neq\emptyset$.
Then the $t$--$u$ path in $T$ contains both $t_{f(x)}$ and $t_i$.
As $B_x$ is connected and meets $V_t$ and $V_u$, it therefore also meets $V_{t_{f(x)}}$ and $V_{t_i}$. This completes the proof of $(*)$.

Since $H$ is $k$-connected, $H$ contains $k$ pairwise distinct vertices $x_1,\dots,x_k$. By $(*)$ we have $V_{t_i}\cap B_{x_j}\neq\emptyset$ for all $i\geq m:=\max\{f(x_\ell):\ell\leq k\}$ and $j\leq k$. Thus $B_{x_j}$ meets the adhesion set $V_{t_m}\cap V_{t_{m+1}}$ for all $j\leq k$, which is a contradiction since the adhesion of $(T,\V)$ is $<k$. Hence $T$ does not contain a directed ray.
\end{proof}

\section{Lemmas on universal graphs without long paths}

In this section, we prove the existence of edge-coloured universal graphs with forbidden paths and finitely many other forbidden substructures. These universal graphs will serve as bricks for constructing universal graphs in later sections.
In the following Lemma, we extend a result of Cherlin and Tallgren \cite{cherlintallgren} to edge-coloured graphs; we follow their original proof. For all $n\in\N$, let $P_n$ denote a fixed path with $n$ edges.

\begin{lemma}\label{pn-free}
Let $c\in\N$ and let $\X$ be a finite set of finite connected $c$-graphs. Suppose that for some $n\in\N$, every $c$-edge-coloured version of $P_n$ is contained in $\X$, or in short, $\{P_n\}^c\subseteq\X$. Then there exists a universal $\X$-free $c$-graph.
\end{lemma}

\begin{proof}
We show by induction on $n$ that the following stronger statement holds for all $c,d,n\in\N$:
\begin{itemize}
\item[$(\dagger)$]
Let $\X$ be any finite set of finite connected $(c,d)$-graphs with $\{P_n\}^{(c,d)}\subseteq\X$ and denote by $\G$ the class of all $\X$-free $(c,d)$-graphs.
Then there is a subclass $\Gamma(\X,c,d)\subseteq\G$ containing $(c,d)$-graphs of only finitely many isomorphism types, such that every connected $(c,d)$-graph $G\in\G$ is an induced subgraph of some $(c,d)$-graph $\Gamma_G \in \Gamma(\X,c,d)$.
\end{itemize}
Then we pick one representative of each isomorphism type in $\Gamma(\X,c,d)$ and take the disjoint union of $\aleph_0$ copies of each representative. This $(c,d)$-graph is clearly universal in $\G$, which completes the proof of the Lemma when we set $d=1$.
Therefore, it suffices to prove $(\dagger)$.

If $n=1$, we set $\Gamma(\X,c,d) := \{\emptyset\}$. Now suppose that $n > 1$. For every connected $(c,d)$-graph $G\in\G$, we will define a $(c,d)$-graph $\Gamma_G\in\G$ containing $G$ as an induced subgraph. In the end, we will see that $\Gamma(\X,c,d):=\{\Gamma_G:G\in\G\}$ only contains finitely many $(c,d)$-graphs up to isomorphism, which proves $(\dagger)$. If $G$ does not contain a path of length $n-1$, then there is $\Gamma_G\in\Gamma(\{P_{n-1}\}^{(c,d)}\cup\X,c,d)\subseteq\G$ containing $G$ as an induced subgraph by induction. Otherwise, let $P$ be a path of length $n-1$ in $G$; we consider $P$ as a $(c,d)$-graph with colourings inherited from $G$.

For every $(c,d)$-graph $H$ we define $t(H)$, which is the graph $H-V(P)$ \footnote{$V(P)$ need not be a subset of $V(H)$.} with edge-colouring induced by $H$ but with a new vertex-colouring: For every vertex $v$ of $H-V(P)$, we consider the following information:
\begin{itemize}
\item the colour of $v$ in $H$,
\item the set of neighbours of $v$ on $P$,
\item the colour of the $v$--$p$ edge for each neighbour $p$ of $v$ on $P$.
\end{itemize}
Then there is a finite number $d'$ (independent of $v$ and $H$) such that there are exactly $d'$ combinations in which these items can occur for each $v$.
Thus we are able to encode this information in a $d'$-vertex-colouring of $t(H)$. Note that for every $(c,d')$-graph $H'$ with $V(H')\cap V(P)=\emptyset$, there is a unique $(c,d)$-graph $H$ containing $P$ as a subgraph such that $t(H)=H'$. We write $\overset{\circ}{t}(H'):=H$.

Now we begin with the construction of $\Gamma_G$. Let $\C$ be the set of components of $t(G)$ and consider a component $C\in\C$. Set $k:=\max\{|X|:X\in\X\}$. We define $\Y_C$ to be the set of all $(c,d')$-graphs $Y$ whose vertex set is a subset of $\{1,\ldots,k\}$ such that $C$ is $Y$-free. Next, it is easy to see that every two longest paths in a connected graph share a vertex. Therefore, since $C$ is disjoint from $P$, which is a longest path in $G$, it follows that $\{P_{n-1}\}^{(c,d')}\subseteq\Y_C$. Hence $(\dagger)$ gives us a finite set $\Gamma(\Y_C,c,d')$ of $(c,d')$-graphs by induction and we choose a $(c,d')$-graph $C^*\in\Gamma(\Y_C,c,d')$ such that $C$ is an induced subgraph of $C^*$.

Let $\D$ be a set containing $\aleph_0$ disjoint copies of each $(c,d')$-graph $D$ which is contained at least $k$ times in the family $(C^*)_{C\in\C}$. We define $\overline{t(G)}$ to be the disjoint union $\bigsqcup_{C\in\C}C^*\sqcup\bigcup\D$. We do not change the colouring of any graph in this union so that $\overline{t(G)}$ is again a $(c,d')$-graph. Without loss of generality, we assume that $V(\overline{t(G)})\cap V(P)=\emptyset$ so that we can define $\Gamma_G:=\overset{\circ}{t}(\overline{t(G)})$.

Then $G$ is an induced subgraph of $\Gamma_G$: Indeed, $t(G)$ is an induced subgraph of $\overline{t(G)}$ and thus $G=\overset{\circ}{t}(t(G))$ is an induced subgraph of $\Gamma_G=\overset{\circ}{t}(\overline{t(G)})$.
Next, we show that $\Gamma_G\in\G$. Clearly, $\Gamma_G$ is countable. Now suppose for a contradiction that a copy $X$ of a $(c,d)$-graph from $\X$ is a subgraph of $\Gamma_G$. Then also $t(X)$ is a subgraph of $t(\Gamma_G)=\overline{t(G)}$. Since $X$ and therefore $t(X)$ has at most $k$ vertices, $t(X)$ is also a subgraph of $t(G)$. However, this implies that $X$ is a subgraph of $G$, contradicting that $G\in\G$.

It is left to show that there are only $(c,d)$-graphs of finitely many isomorphism types contained in $\{\Gamma_G:G\in\G\}$. In the case that $G$ does not contain a path of length $n-1$, there are only finitely many possible isomorphism types for $\Gamma_G$ by assumption on $\Gamma(\{P_{n-1}\}^{(c,d)}\cup\X,c,d)$. If $G$ does contain a path $P$ of length $n-1$, the same is true because there are only finitely many possible colourings of $P$, only finitely many possible choices for the map $t$, only finitely many possible sets $\Y_C$, only graphs of finitely many isomorphism types in $\Gamma(\Y_C,c,d')$, and only finitely many multiplicities (namely $0,1,\ldots,k-1,\aleph_0)$ in which the $(c,d')$-graphs from each set $\Gamma(\Y_C,c,d')$ can occur in $\overline{t(G)}$.
\end{proof}

We now derive from Lemma \ref{pn-free} that there is a universal graph with any excluded finite connected minor if we additionally exclude a path:

\begin{corollary}\label{forbidden-minor}
Let $c,n\in\N$ and let $X$ be a finite connected $c$-graph. Then there exists a universal $\{X \cup P_n\}^c$-minor-free $c$-graph.
\end{corollary}

\begin{proof}
Let $\M$ be the class of all models $X'$ of $X$ not containing a path of length $n$ such that all branch sets induce finite trees in $X'$ with less than $|X|$ leafs.
Further, let $\G$ be the class of all $\{X \cup P_n\}$-minor-free graphs and $\H$ the class of all $\M \cup \{P_n\}$-free graphs. Our first step is to show that $\G=\H$.

Clearly, $\G\subseteq\H$. Conversely, suppose that $G$ is a graph that is not contained in $\G$. If $G$ contains a path of length $n$ as a minor, then it is also contains a path of length $n$ as a subgraph and consequently $G$ cannot be contained in $\H$.
Now suppose that $P_n$ is not a minor of $G$ but $X$ is, and let $X'$ be a model of $X$ in $G$.
For any $y\in V(X)$, let $B_y$ denote the branch set in $X'$ corresponding to $y$.
Now fix for every edge $yz\in E(X)$ a $B_y$--$B_z$ edge $e_{yz}$ in $X'$.
Since $X'[B_y]$ is connected, there is a tree $T_y$ in $X'[B_y]$ such that all leaves of $T_y$ are endvertices of edges $e_{yz}$ for $z\in N_X(y)$. Now we replace every subgraph of the form $X'[B_y]$ of $X'$ with the tree $T_y$ to obtain a subgraph of $X'$ which lies in $\M$. It follow that $G$ is not contained in $\H$, which completes the proof that $\H\subseteq\G$.

It is left to show that there are only graphs of finitely many isomorphism types contained in $\M$ and thus in $(\M \cup \{P_n\})^c$. Then by Lemma \ref{pn-free}, there is a universal $(\M \cup \{P_n\})^c$-free $c$-graph, i.e. a universal $c$-graph in $\H^c$. Since $\G = \H$, it follows that this graph is also universal in $\G^c$, i.e. the class of all $\{X \cup P_n\}^c$-minor-free graphs, completing the proof.

To see that $\M$ contains graphs of only finitely many isomorphism types, it suffices to show that the graphs in $\M$ have finitely bounded size. To accomplish that, it suffices, in turn, to show that the branch sets of all models of $X$ contained in $\M$ have finitely bounded size. So let $X'$ be any such model and let $T$ denote a graph induced by one of the branch sets. The vertices of $T$ have less than $|X|$ neighbours since $T$ is a finite tree with less than $|X|$ leaves. As $T$ has maximum degree less than $|X|$ and does not contain a path of length $n$, we obtain an upper bound on the size of $T$, as desired.
\end{proof}

\section{Unavoidable minors of 2-connected graphs containing long paths}

For all integers $n,m\geq 3$, let $C_n$ denote a fixed cycle of length $n$ and $C_{n,m}$ a fixed graph consisting of a cycle of length $m$ and a cycle of length $n$ which intersect in exactly one edge. In this section, we recall/prove that every 2-connected graph containing arbitrarily long paths contains $C_n$ and $C_{n,m}$ as minors. These results will be used in the next section for showing the existence of universal $C_n$-minor-free and $C_{n,m}$-minor-free graphs.

\begin{lemma}[\cite{diestel} Chapter 1, Exercise 3]\label{cycle}
Every 2-connected graph $G$ containing a path of length $n^2$ also contains a cycle of length at least $n$ for all $n\geq 3$.
\end{lemma}

\begin{lemma}\label{2-con}
Let $n,m\geq 3$ and let $X$ be either the graph $C_n$ or the graph $C_{n,m}$. Then $X$ is a minor of $G$ for any 2-connected graph $G$ containing arbitrarily long paths.
\end{lemma}

\begin{proof}
If $X=C_n$, the claim follows from Lemma \ref{cycle}. Therefore, suppose that $X=C_{n,m}$. Since $G$ contains arbitrarily long paths, there is a cycle $D^1$ in $G$ of length at least $2n$ and a cycle $D^2$ of length at least $|D^1|\cdot m$ by Lemma \ref{cycle}.

If $D^1\cap D^2=\emptyset$, consider two disjoint $D^1$--$D^2$ paths $P^1$ and $P^2$ in $G$, which exists since $G$ is 2-connected. As $|D^1|\geq 2n$ and $|D^2|\geq 2m$, it is easy to find a subdivision of $C_{n,m}$ in $G':=D^1\cup D^2\cup P^1\cup P^2$. If $D^1\cap D^2=\{v\}$ for some vertex $v$, then let $P^1:=v$ be a path of length 1. Since $G$ is 2-connected, there is a $D^1$--$D^2$ path $P^2$ in $G-v$ and again there is a subdivision of $C_{n,m}$ in $G'$. Finally, suppose that $|D^1\cap D^2|\geq 2$ and note that there are at most $|D^1|$ many $D^1$-paths (i.e. paths with both endvertices in $D^1$ but no inner vertices in $D^1$) contained in $D^2$. Therefore, there must be a $D^1$-path $P$ in $D^2$ of length at least $m$ since $|D^2|\geq |D^1|\cdot m$. Then $D^1\cup P$ is a subdivision of $C_{n,m}$ in $G$.
\end{proof}

Note that $C_n$ and $C_{n,m}$ for $n,m\geq 3$ are the only graphs for which Lemma \ref{2-con} is true. Indeed, $C_n$ and $C_{n,m}$ are the only 2-connected graphs which are minors of both graphs in Figure \ref{2-connected-graphs}. To see this, first notice that any 2-connected minor of the right graph in Figure \ref{2-connected-graphs} is a union of cycles all having precisely one edge in common. However, if such a graph is also a minor of the left graph, it can only consist of at most two cycles and thus it is isomorphic to $C_n$ or $C_{n,m}$ for some $n,m$.

\begin{figure}[ht]
\centering
\includegraphics[scale=0.25]{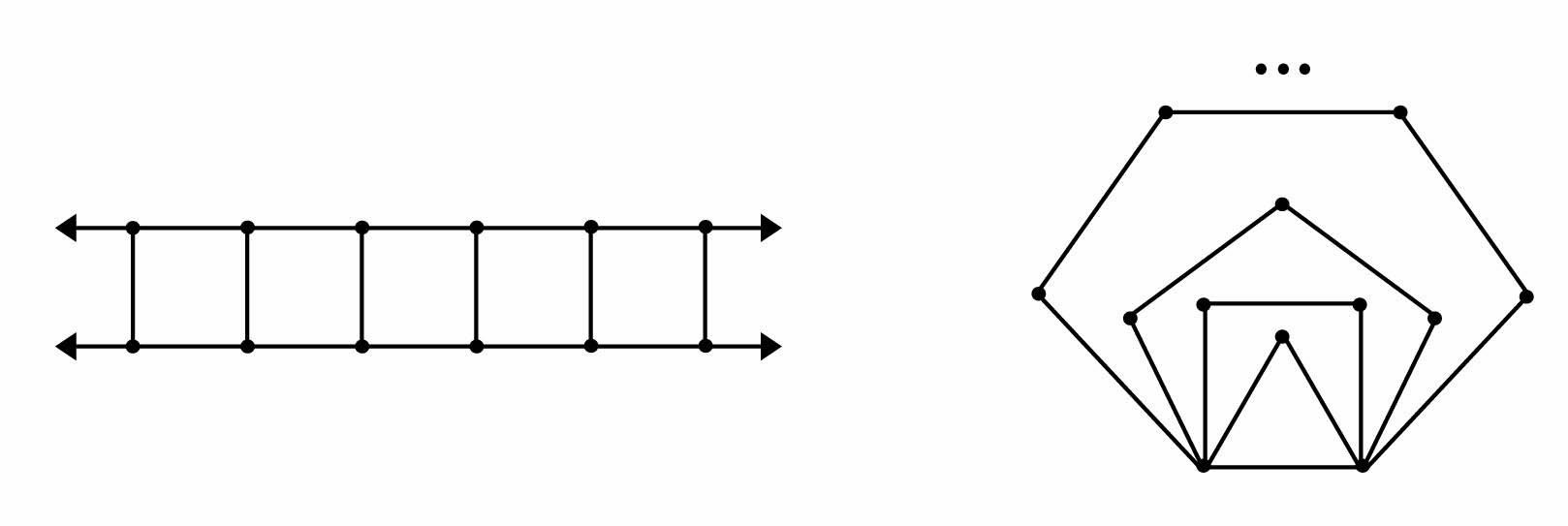}
\caption{Two 2-connected graphs containing arbitrarily long paths.}
\label{2-connected-graphs}
\end{figure}

\section{Universal graphs with forbidden cycle minors}\label{sectionforbiddencycleminors}

\begin{theorem}\label{forbiddencycle}
Let $n,m\geq 3$ and let $X$ be either the graph $C_n$ or the graph $C_{n,m}$. Then there exists a universal $X$-minor-free graph.
\end{theorem}

\begin{proof}
It suffices to construct an $X$-minor-free graph $\Gamma$ such that every connected $X$-minor-free graph is isomorphic to an induced subgraph of $\Gamma$. Then the disjoint union of $\aleph_0$ copies of $\Gamma$ is universal $X$-minor-free.

For all $n\in\N$, let $\Delta_n$ be a universal $\{X,P_n\}$-minor-free graph, which exists by Corollary \ref{forbidden-minor}. For the construction of $\Gamma$, we recursively define a sequence of graphs $\Gamma_0\subseteq\Gamma_1\subseteq\Gamma_2\subseteq\dots$. Let $\Gamma_0$ consist of a single vertex and suppose that we have defined $\Gamma_0,\ldots,\Gamma_i$. To construct $\Gamma_{i+1}$, we begin with the graph $\Gamma_i$. For every vertex $v\in V(\Gamma_i)$, for every integer $n\in\N$, and for every vertex $w\in V(\Delta_n)$, we add $\aleph_0$ disjoint copies of $\Delta_n$ to $\Gamma_i$ and identify $v$ with the vertex corresponding to $w$ in each copy of $\Delta_n$. Finally, we define
$$\Gamma:=\bigcup_{i\in\N}\Gamma_i.$$

Then $\Gamma$ is clearly countable. Furthermore, since $X$ is 2-connected but $\Gamma$ consists of $X$-minor-free graphs which are pasted together along cutvertices, it follows from Lemma \ref{lem:minor_lives_in_part} that $\Gamma$ is $X$-minor-free.

It is left to show that every connected $X$-minor free graph $G$ is isomorphic to an induced subgraph of $\Gamma$. Let $B_0,B_1,B_2,\ldots$ be an enumeration of all \emph{blocks of $G$} (i.e. maximal connected subgraphs of $G$ without cutvertices) such that $B^j:=\bigcup_{i\leq j}B_i$ is connected for all $j\in\N$. Note that each $B^j$ intersects $B_{j+1}$ in exactly one vertex. For every block $B_i$ of $G$, there is an integer $n(i)\in\N$ such that $B_i$ does not contain a path of length $n(i)$ by Lemma \ref{2-con}. Hence $B_i$ is an induced subgraph of $\Delta_{n(i)}$. Thus we can find a copy of $G$ in $\Gamma$ by recursively embedding each block $B_i$ of $G$ into a suitable copy of $\Delta_{n(i)}$ in $\Gamma$.
\end{proof}

\section{Unavoidable minors of 3-connected graphs containing long paths}

For all integer $k\geq 3$, let $W_k$ be the wheel on $k+1$ vertices. Our aim in this section is to prove the following analogue to Lemma \ref{2-con} for 3-connected graphs:

\begin{theorem}\label{3-con}
There is a function $f:\N\to\N$ such that every 3-connected graph containing a path of length $f(k)$ contains $W_k$ as a minor.
\end{theorem}

Theorem \ref{3-con} can be derived from arguments in \cite{oporowskioxleythomas} by Oporowski, Oxley and Thomas, and a few additional considerations. We will now briefly summarize the parts of \cite{oporowskioxleythomas} that we need. We begin by defining the graphs $D_k$, $L_k$, $O_k$ and $M_k$ for all integers $k\geq 3$ (Figure \ref{typicalgraphs}). Let $D_k$ be the graph consisting of a cycle of length $k$ together with two additional vertices which are both adjacent to all vertices on the cycle but not to each other. Let $L_k$ be the "ladder" consisting of two paths $v_1v_2\ldots v_k$ and $w_1w_2\ldots w_k$ with the edges $v_iw_i$ added in for all $i\leq k$. The graph $O_k$ is obtained from $L_k$ by adding the edges $v_1v_k$ and $w_1w_k$ and $M_k$ is obtained from $L_k$ by adding the edges $v_1w_k$ and $w_1v_k$. 

\begin{figure}[ht]
\centering
\includegraphics[scale=0.18]{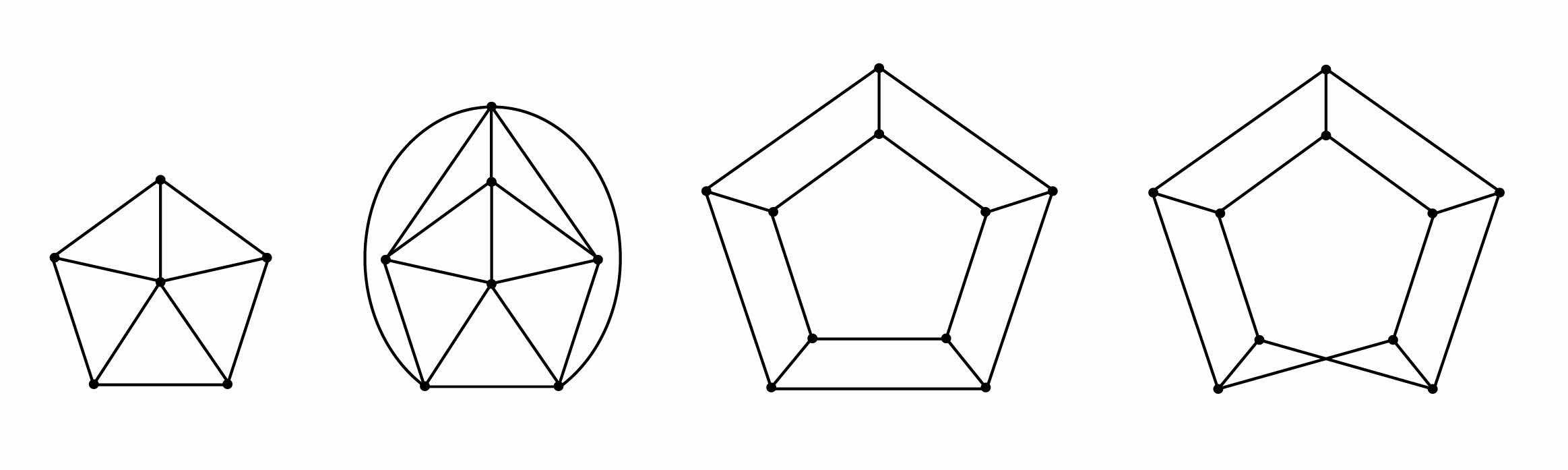}
\caption{The graphs $W_5$, $D_5$, $O_5$ and $M_5$ (from left to right).}
\label{typicalgraphs}
\end{figure}

\begin{theorem}[\cite{oporowskioxleythomas}, (1.4)]\label{typical4}
There is a function $f_4:\N_{\geq 3}\to\N$ such that every finite 4-connected graph $G$ with at least $f_4(k)$ vertices contains $D_k$, $O_k$, $M_k$ or $K_{4,k}$ as a minor.
\end{theorem}

\begin{proof}[Proof outline]\renewcommand{\qedsymbol}{}
Let $k\in\N$ and assume that $O_k$ is not a minor of $G$.
By \cite{robertsonseymour}, there is an integer $w(k)$ such that the tree-width of $G$ is less than $w(k)$. This information is used for constructing a tree-decomposition $(T,\V)$ of $G$ of width less than $w(k)$ satisfying certain additional properties, which we will not name here. Next, it is derived that there are functions $d,p:\N\to\N$ (independent of $G$ and $(T,\V)$) with the following properties. If $T$ contains a node of degree $d(k)$, then $K_{3,k}$ is a minor of $G$, and if $T$ contains a path of length $p(k)$, then $D_k$ or $M_k$ is a minor of $G$.

If $|V(G)|$ is large, then $|V(T)|$ must be large, too, as the width of $(T,\V)$ is bounded. Hence, by choosing $f_4(k)$ sufficiently large, we can ensure that $T$ contains a node of degree $d(k)$ or a path of length $p(k)$, which proves the theorem.
\end{proof}

\begin{theorem}[\cite{oporowskioxleythomas}, (1.3)]\label{typical3}
There is a function $f_3:\N_{\geq 3}\to\N$ such that every 3-connected graph $G$ with at least $f_3(k)$ vertices contains the wheel $W_k$ or the complete bipartite graph $K_{3,k}$ as a minor.
\end{theorem}

\begin{proof}
Let $f_3(k):=f_4(k+1)$. We define the 4-connected graph $G^+$ by adding a vertex $v$ to $G$ and edges from $v$ to all vertices of $G$. By Theorem \ref{typical4}, $G^+$ contains one of $D_{k+1}$, $O_{k+1}$, $M_{k+1}$ or $K_{4,k+1}$ as a minor. Note that for any vertex $v$, the wheel $W_k$ is a minor of $D_{k+1} - v$, $O_{k+1} - v$, and $M_{k+1} - v$, and $K_{3,k}$ is a minor of $K_{4,k+1} - v$. Hence $G$ contains either $W_k$ or $K_{3,k}$ as a minor, as desired.
\end{proof}

By combining the results above from \cite{oporowskioxleythomas} with Lemma \ref{longpath}, we can prove Theorem \ref{3-con} for finite graphs:

\begin{corollary}\label{finite}
There is a function $f:\N\to\N$ such that every finite 3-connected graph containing a path of length $f(k)$ contains $W_k$ as a minor.
\end{corollary}

\begin{proof}
Let $f(k):=\ell_{w(k+1)}(p(k+1))$ for all $k\in\N$ where $w$ and $p$ are the functions from the proof outline of Theorem \ref{typical4} and $\ell_{w(k+1)}$ is the function from Lemma \ref{longpath}.
Let $G$ be a finite 3-connected graph containing a path of length $f(k)$ and consider the graph $G^+$ from the proof of Theorem \ref{typical3}.
If $O_{k+1}$ is a minor of $G^+$, then $W_k$ is a minor of $G$ as desired. Otherwise, let $(T,\V)$ be the tree-decomposition of $G^+$ of width less than $w(k+1)$ from the proof of Theorem \ref{typical4}. By Lemma \ref{longpath}, $T$ contains a path of length $p(k+1)$. Therefore, $G^+$ contains one of $D_{k+1}$ or $M_{k+1}$ as a minor, as we have seen in the proof outline of Theorem \ref{typical4}. Hence $W_k$ is a minor of $G$ as in the proof of Theorem \ref{typical3}.
\end{proof}

Next, we need to understand another result from \cite{oporowskioxleythomas} about unavoidable minors of infinite 3-connected graphs. Let $R_2$ be the graph consisting of a ray and two additional non-adjacent vertices which are both adjacent to all vertices of the ray.

\begin{theorem}[\cite{oporowskioxleythomas},  (5.2)]\label{typicalinf}
Every infinite 3-connected graph $G$ contains $K_{3,\omega}$ or $R_2$ as a minor.
\end{theorem}

\begin{proof}[Proof outline]\renewcommand{\qedsymbol}{}
If $G$ contains a ray, then one can show that $R_2$ is a minor of $G$. If $G$ is rayless, one can show that $K_{3,\omega}$ is a minor of $G$.
\end{proof}

We can now complete the proof of Theorem \ref{3-con}, using the following result of Halin:

\begin{lemma}[\cite{halin}, Lemma 3]\label{rayless}
If $G$ is a rayless $k$-connected graph and $U\subseteq V(G)$ is finite, then there exists a finite $k$-connected subgraph $H$ of $G$ with $U\subseteq V(H)$.
\end{lemma}

\begin{proof}[Proof of Theorem \ref{3-con}]
Let $f$ be the function from Corollary \ref{finite} and let $G$ be a 3-connected graph containing a path $P$ of length $f(k)$. If $G$ contains a ray, then $G$ contains $R_2$ as a minor by the proof outline of Theorem \ref{typicalinf}. Since $R_2$ contains $W_k$ as a minor, also $G$ does. Therefore, suppose that $G$ is rayless. By Lemma \ref{rayless}, there is a finite 3-connected subgraph $G'$ of $G$ containing $P$. Therefore, $W_k$ is a minor of $G'$ by Corollary \ref{finite} and hence also of $G$.
\end{proof}

\section{Universal graphs with forbidden wheel minors}\label{forbiddenwheelminors}

In this section, we show that there exists a universal $W$-minor-free graph $\Gamma$ for every wheel $W$; the proof is based on similar ideas as the proof of Theorem \ref{forbiddencycle}.
Instead of decomposing graphs into their 2-connected blocks as in the proof of Theorem \ref{forbiddencycle}, we will decompose graphs into 3-connected graphs and cycles.
For this, we use the following theorem of Richter \cite{richter}, which generalises a well-known result of Tutte \cite{tutte} to infinite graphs.

\begin{lemma}[\cite{richter}]\label{tree-decomposition}
Every graph $G$ has a tree-decomposition of adhesion at most 2 in which all torsos are 3-connected graphs, cycles, or isomorphic to $K_1$ or $K_2$. Additionally, if $vw$ is any edge that is contained in the torso of a part $V_t$ but not in $G$, then there is a $v$--$w$ path in $G$ with no inner vertices in $V_t$.
\end{lemma}

\begin{theorem}\label{forbiddenwheel}
Let $X$ be any finite 3-connected graph that is a minor of every 3-connected graph containing arbitrarily long paths. Then there exists a universal $X$-minor-free graph.
\end{theorem}

\begin{proof}
The proof is similar to the proof of Theorem \ref{forbiddencycle}. Again, it suffices to construct an $X$-minor-free graph containing all connected $X$-minor-free graphs as induced subgraphs.

For all $n\in\N$, let $\G_n$ denote the class of all (countable) graphs not containing $X$ and $P_n$ as a minor. By Lemma \ref{forbidden-minor}, there is a universal 2-graph $\Delta_n$ in $\G_n^2$.
We recursively define a sequence of 2-graphs $\Gamma^*_0\subseteq\Gamma^*_1\subseteq\Gamma^*_2\subseteq\ldots$. Let $\Gamma^*_0$ consist of a single vertex and suppose that we have defined $\Gamma^*_0,\ldots,\Gamma^*_i$.
For the construction of $\Gamma^*_{i+1}$, we begin with the 2-graph $\Gamma^*_i$. For every vertex $v\in V(\Gamma^*_i)$, for every $n\in\N$, and for every vertex $w\in V(\Delta_n)$, we add $\aleph_0$ disjoint copies of $\Delta_n$ and identify $v$ with the vertex corresponding to $w$ in each copy of $\Delta_n$. Similarly, for every edge $e\in E(\Gamma^*_i)$, for every $n\in\N$, and for every edge $f\in E(\Delta_n)$ of the same colour as $e$, we add $\aleph_0$ disjoint copies of $\Delta_n$ and identify $e$ with the edge corresponding to $f$ in each copy of $\Delta_n$. We choose the edge-colouring of $\Gamma^*_{i+1}$ which is inherited from $\Gamma^*_i$ and from the copies of $\Delta_n$, which completes the construction of $\Gamma^*_{i+1}$. We define
$$\Gamma^*:=\bigcup_{i\in\N}\Gamma^*_i$$
and consider the edge-colouring of $\Gamma^*$ which extends the edge-colourings of each $\Gamma^*_i$. Let $\Gamma$ be the subgraph of $\Gamma^*$ which is induced by all edges of colour 0 in $\Gamma^*$.

We show that the graph $\Gamma$ is universal $X$-minor-free. Firstly, $\Gamma^*$ (considered as an uncoloured graph) and therefore its subgraph $\Gamma$ is $X$-minor-free: It is clear that $\Gamma^*$ is countable. Furthermore, by construction there is a tree-decomposition of $\Gamma^*$ of adhesion 2 with complete adhesion sets such that all graphs which parts induce in $\Gamma^*$ are copies of some $\Delta_n$. Since each $\Delta_n$ is $X$-minor-free, it follows from Lemma \ref{lem:minor_lives_in_part} that also $\Gamma^*$ is $X$-minor-free.

It is left to show that every connected $X$-minor-free graph $G$ is isomorphic to an induced subgraph of $\Gamma$. Consider the tree-decomposition $(T,\V)$ of $G$ from Lemma \ref{tree-decomposition} and let $G^*$ denote the graph obtained from $G$ by adding edges to make all adhesion sets complete. We define a 2-edge-colouring of $G^*$ by colouring all edges in $E(G^*)\cap E(G)$ with 0 and all edges in $E(G^*)\setminus E(G)$ with 1. We will show that the 2-graph $G^*$ is isomorphic to an induced subgraph of the 2-graph $\Gamma^*$, then it follows that the graph $G$ is isomorphic to an induced subgraph of $\Gamma$.

Consider an arbitrary node $t\in V(T)$. Since the adhesion of $(T,\V)$ is 2, by the additional statement in Lemma \ref{tree-decomposition}, the torso $G^*[V_t]$ (viewed as an uncoloured graph) is a minor of $G$. Thus $X$ is not a minor of $G^*[V_t]$, as otherwise $X$ would also be a minor of $G$. Since $G^*[V_t]$ is either 3-connected, a cycle, or isomorphic to $K_1$ or $K_2$, there is an integer $n(t)\in\N$ such that $G^*[V_t]$ does not contain a path of length $n(t)$ by the choice of $X$. Hence the 2-graph $G^*[V_t]$ is isomorphic to an induced subgraph of $\Delta_{n(t)}$. We can now prove that $G^*$ is isomorphic to an induced subgraph of $\Gamma^*$ by fixing an enumeration $t_0,t_1,t_2,\dots$ of $V(T)$ such that the set $\{t_i:i<j\}$ is connected in $T$ for all $j\in\N$, and then recursively embedding each subgraph $G^*[V_{t_i}]$ into a suitable copy of $\Delta_{n(t_i)}$ in $\Gamma^*$.
\end{proof}

\begin{corollary}
For all integers $k\geq 3$, there is a universal $W_k$-minor-free graph.
\end{corollary}

\begin{proof}
The claim follows from Theorem \ref{3-con} and Theorem \ref{forbiddenwheel}.
\end{proof}

\begin{remark}
We do not know whether there are any finite 3-connected graphs $X$ other than wheels which occur as minors in all 3-connected graphs containing arbitrarily long paths. If there are any such graphs $X$, then there is a universal $X$-minor-free graph by Theorem \ref{forbiddenwheel}.    
\end{remark}

\begin{remark}
A short explanation why it is necessary to consider edge-coloured graphs in the proof of Theorem \ref{forbiddenwheel}: The graph $\Gamma^*$ (without its edge-colouring) is already weakly universal $X$-minor-free. However, it is not strongly universal since all its adhesion sets are complete. So an $X$-minor-free graph for which the adhesion sets of the tree-decomposition from Lemma \ref{tree-decomposition} are not complete, might only embed as a subgraph into $\Gamma^*$ but not as an induced subgraph. Fixing this by gluing together the graphs $\Delta_n$ also along incomplete adhesion sets is tricky, since this might create copies of $X$. The purpose of the edges of colour 1 in the proof is to produce incomplete adhesion sets by deleting these edges.
\end{remark}

\begin{remark}
Theorem \ref{thm_main_universal} is also true for the class of all $\kappa$-sized $X$-minor-free graphs for any infinite cardinal $\kappa$. Only minor changes to our proof are necessary to establish that.
\end{remark}

\medskip

\bibliographystyle{amsplain}
\bibliography{bibl}

\end{document}